\documentclass[11pt]{article}
\usepackage{amsthm}
\usepackage{amsmath,amsthm,amsfonts,amssymb,mathrsfs,bm,graphicx,stmaryrd,dsfont}
\usepackage[colorlinks=true,linkcolor=blue]{hyperref}
\usepackage[usenames,dvipsnames]{color}
\usepackage{fullpage}
\usepackage[american]{babel}
\usepackage[varg]{pxfonts}
\usepackage{prettyref}
\usepackage{microtype}
\newtheorem{theorem}{Theorem}[section]

\newtheorem{lemma}[theorem]{Lemma}
\newtheorem{remark}[theorem]{Remark}

\newtheorem{claim}[theorem]{Claim}

\newcommand{\R}{\mathbb{R}}
\newcommand{\C}{\mathbb{C}}
\newcommand{\dist}{\mathsf{dist}}
\newcommand{\oi}{{1\rightarrow\infty}}
\renewcommand{\P}{\mathbb{P}}

\newcommand{\F}{\mathcal{F}}
\newcommand{\B}{\mathcal{B}}
\newcommand{\e}{\varepsilon}
\renewcommand{\bm}[1]{\begin{bmatrix} #1\end{bmatrix}}

\title{On Non-localization of Eigenvectors of High Girth Graphs}
\author{Shirshendu Ganguly\thanks{sganguly@berkeley.edu;  Supported by a Miller Research Fellowship}\\UC Berkeley \and Nikhil Srivastava\thanks{srivastn@berkeley.edu; Supported by NSF Grant CCF-1553751 and a Sloan Research Fellowship.}\\{UC Berkeley}}
\date{\today}
\begin{document}
\maketitle
\begin{abstract} {We prove improved bounds on how localized an
	eigenvector of a high girth regular graph can be, and present examples
	showing that these bounds are close to sharp. This study was
	initiated by Brooks and Lindenstrauss \cite{bl} who relied on the
	observation that certain suitably normalized averaging operators on
	high girth graphs are hyper-contractive and can be used to approximate
	projectors onto the eigenspaces of such graphs. Informally, their
	delocalization result in the contrapositive states that for any $\e \in
	(0,1)$ and positive integer $k,$ if a $(d+1)-$regular graph has an
	eigenvector which supports $\e$ fraction of the $\ell_2^2$ mass on a
	subset of $k$ vertices, then the graph must have a cycle of size
	$\log_{d}(k)/\e^2)$, up to multiplicative universal  constants and additive logarithmic terms in $1/\e$. In this paper, we improve the upper bound to
	$\log_{d}(k)/\e$ up to similar logarithmic correction terms; and present a construction showing a lower bound
	of $\log_d(k)/\e$ up to multiplicative constants. Our construction is probabilistic and
	involves gluing together a pair of trees while maintaining high girth
	as well as control on the eigenvectors and could be of independent
	interest. }
 \end{abstract}
 
\section{Introduction}
Spectral graph theory studies graphs via associated linear operators such as
the Laplacian and the adjacency matrix.  While the extreme eigenvectors of these
operators are relatively well-understood and correspond to sparse cuts and colorings, much
less is known about the combinatorial meaning of the interior eigenvectors.
Most of the literature about them falls into two categories:\\

\noindent 1. {\bf Analysis of eigenvectors of random graphs.} For example, Dekel,
Lee, Linial  \cite{dekel} prove that any eigenvector of  a
{dense} random graph has a bounded number of nodal domains
i.e., connected components where the eigenvector does not
change sign. Following a sequence of results by various authors, in a recent
breakthrough work Bauerschmidt,  Huang, Yau \cite{yau}, among
various other things, show that with high probability, any `bulk'
eigenvector $v$ of a random regular graph with $n$ vertices and a large enough but fixed constant degree, is $\ell_{\infty}$
delocalized in the following sense: $$||v||_\infty\le
\frac{\log^C(n)}{\sqrt{n}}||v||_2,$$ where $||\cdot||_2$, and
$||\cdot||_\infty$ denote the usual $\ell_2$ and
$\ell_{\infty}$ norms respectively and $C$ is a constant.  For a more precise
statement see Theorem 1.2 in \cite{yau}. 
In another line of work, Backhausz and Szegedy  \cite{szegedy}  establish Gaussian behavior of the entry distribution of eigenvectors of random regular graphs by studying factors of i.i.d. processes on the regular infinite tree.

In all of these works the randomness of the model is used
		heavily, and weaker notions of delocalization are also
		considered (see e.g. \cite{geisinger}). We refer the reader to
		\cite{vusurvey} for a survey of recent developments on
		delocalization of eigenvectors of random matrices.\\

 \noindent 2. A parallel story based on {\bf asymptotic analysis of sequences
 of deterministic graphs}. The driving force for this is the so called  Quantum
 Unique Ergodicity (QUE) conjecture by  Rudnick and Sarnak
 \cite{rudnicksarnak}. The QUE conjecture states that on any compact negatively
 curved manifold all high energy eigenfunctions of the Laplacian
 equi-distribute. The conjecture is still widely open having been  verified in
 only a few cases; perhaps most notably for the Hecke orthonormal basis on an
 arithmetic surface by Lindenstrauss \cite{lindenstrauss}.
Brooks-Lindenstrauss \cite{bl} initiated the study of graph-theoretic analogues of this conjecture. The analogue of negatively
curved manifolds are high girth regular graphs --- the girth is defined as the length of the shortest cycle in a graph.
		Subsequently, Anantharaman and Le-Masson \cite{anantharaman}
		proved an asymptotic version of quantum ergodicity for regular expanders
		which converge  (in the Benjamini-Schramm local topology) to
		the infinite $d-$regular tree.  

The starting point of this paper is the beautiful result of \cite{bl}. Since the
statement is a bit technical and could be hard to parse at first read we
first explain the content informally in words.  The theorem roughly says that
if a graph does not have many short cycles, then eigenvectors cannot localize
on small sets: for any eigenvector, any subset of the vertices representing a
fraction of the $\ell_2^2$ mass must have size $n^\delta$ for some $\delta$
depending on the fraction. The condition of not having many cycles is articulated
as hyper-contractivity (i.e., control of $\|\cdot\|_{p\rightarrow q}$ norms for some $p<q$,  {a more formal definition is added at the beginning of Section \ref{sec12ub}})
of certain spherical mean operators on the graph, where $\|\cdot\|_{p\rightarrow q}$denotes the norm of the naturally associated operator from $\ell_p$ to $\ell_{q}.$
\begin{theorem}[\cite{bl}]\label{thm:bl}
	Suppose $G=(V,E)$ is a $(d+1)-$regular graph with adjacency matrix $A$. Let
	$$S_n(f)(x):=\frac1{d^{n/2}}\sum_{\dist(x,y)=n} f(y)$$
	and suppose
	\begin{equation}\label{eqn:sncondition}\|S_n\|_{p\rightarrow q}\le Cd^{-\alpha n}\end{equation}
	for all $n\le N$, for some $1\le p\le q\le \infty$ and $\alpha\in [0,1]$. Then for any normalized  {$\ell_2$} eigenvector $v=(v_x)_{x\in V},$
	of $A$ and $S\subset V$ with $\|v_S\|_2^2:=\sum_{x\in S}v^2_x\ge \e,$
	$$|S|\ge \Omega_{d,C,\alpha}(\e^{\frac{2p}{2-p}}d^{\delta N}),$$
	where $\delta=2^{-7}\frac{\alpha p}{2-p}\e^2$ and  {the constant in the $\Omega$ notation depends on the parameters $d,C,\alpha$.}
\end{theorem}
In particular, the condition \eqref{eqn:sncondition} is
satisfied with $p=1,q=\infty,\alpha=1/2,C=d$ and $N=\lceil g/2\rceil -1$ for a
graph of girth $g$.  Viewed in the contrapositive, the theorem therefore says
that the existence of an eigenvector of $A$ with $\e$ fraction of its mass on $k=|S|$
coordinates implies that the graph must contain a cycle of length
$O(\log_d(k/\e)/\e^2)$. In fact, a close examination of the proof reveals
that it gives an upper bound which varies between $O(\log_d(k/\e)/\e)$ and
$O(\log_d(k/\e)/\e^2)$ depending on the Diophantine properties of the
eigenvalue being considered.

In this paper, we contribute to the understanding of this phenomenon in two ways..
First, we improve the above bound to $O(\log_d(k/\e)/\e)$ for all
eigenvalues of $d+1$-regular\footnote{We work with $(d+1)$-regular rather than $d-$regular graphs to avoid repeatedly writing $d-1$.} graphs, irrespective of the number theoretic
properties of the eigenvalue. The proof involves
replacing the approximation-theoretic component of their proof by a simpler and
more efficient method.  Specifically, we prove the following theorem in Section
2.

\begin{theorem}\label{thm:main} Suppose $G$ is a $(d+1)$-regular graph of girth $g$ and $v$ is a normalized
	eigenvector of the adjacency matrix of $G$. Then any subset $S$ with $\|v_S\|_2^2\ge \e$ must have
	$$|S|\ge \frac{d^{\e g/4}\e}{2d^2}.$$
\end{theorem}
The contrapositive of the above theorem implies that if there exists $\e$ and $k$ and $S$ such that $|S|=k$ and $\|v_S\|_2^2=\e,$ then  
$$g\le \frac{4\log_d(k/\e)+O(1)}{\e}.$$
Before proceeding further some remarks are in order.
\begin{remark}[Choice of Hypercontractive Norms] The paper \cite{bl} works with general $p\rightarrow q$ norms, but
	in this paper we will work solely with the $1\rightarrow\infty$ since it is
	reveals all of the ideas and is easier to interpret combinatorially. Our proof of 
	Theorem \ref{thm:main} can easily be modified to work with $p\rightarrow q$ norms, if desired.
\end{remark}
\begin{remark}[Entropy Bounds from Delocalization] As already observed in \cite[Corollary 1]{brooks}, it is quite
	straightforward to obtain a lower bound on the entropy of an
	eigenvector $v$ from a delocalization result such as Theorem
	\ref{thm:main}, where the entropy of $v$ is $-\sum_{x\in
	V}v_x^2\log_{d}v_x^2.$ 
\end{remark}
\begin{remark}[Tempered and Untempered Eigenvalues]\label{rem:tempered} Eigenvalues of $A$ in the
	interval $[-2\sqrt{d},2\sqrt{d}]$ are referred to as {\em tempered}
	(indicating wave-like behavior) and those outside are called {\em
	untempered} (indicating exponential growth) in the QUE literature.  {It
	is known that for all eigenvalues which are strictly untempered, i.e., bounded away from $\{-2\sqrt{d},2\sqrt{d}\}$, a much stronger delocalization result, with dependence
	roughly $g=O(\log_d(k/\e))$, can be proven using elementary arguments} --- see e.g.  \cite[Page
	59]{brooks} or the arguments of \cite{kahale}. Note that any sequence
	of graphs with girth going to infinity must have a vanishingly small
	fraction of untempered eigenvalues.  We will present bounds for arbitrary
	eigenvalues in this paper, without focusing on the distinction between
	tempered and untempered.  {A slightly more elaborate discussion is presented in Remark \ref{improve43}}.
\end{remark}

Moreover, for every $d\ge 2$, sufficiently large $k$, and $\e\in (0,1)$, we exhibit a $(d+1)-$regular 
graph with a localized eigenvector which has girth at least $\Omega(\log_d(k)/\e)$,
showing that our improved bound is sharp up to an additive $\log(1/\e)$ factor in the numerator, which is negligible
whenever $k=\Omega(1/\e^c)$ for any $c$.   We are able to construct
such eigenvectors for a dense subset of eigenvalues in $(-2\sqrt{d},2\sqrt{d})$. The proof is probabilistic, and
involves gluing together two trees without introducing any short cycles and
while controlling their eigenvectors, which may be of independent interest.
\begin{theorem}\label{thm:lb} 
	For every $d\ge 2$, sufficiently large $k$ and all $\e>0$, there is a finite $(d+1)-$regular graph $G$ with the following properties.
\begin{enumerate} \item $A_G$ has a normalized eigenvector $v$ with eigenvalue $\lambda\in (-2\sqrt{d},2\sqrt{d})$ and
				$$\|v_S\|_2^2=\Omega_\lambda(\e)$$
				for a set $S$ of size $k$, where the implicit constant $C_{\lambda}$ depends on $\lambda$ and is bounded away
				from zero on any subinterval of $(-2\sqrt{d},2\sqrt{d})$,  {i.e., $$\inf_{\lambda \in [-2\sqrt{d}+\delta,2\sqrt{d}-\delta]}C_{\lambda}> 0,$$ for any small enough positive $\delta.$}
		\item $G$ has girth at least 
			$$\Omega\left(\frac{\log_{d}(k)}{\e}\right).$$
	\end{enumerate}
	For every fixed $\e$ (or for every fixed, sufficiently large $k$), the set of eigenvalues attained by the above graphs is dense in $(-2\sqrt{d},2\sqrt{d})$.
\end{theorem}
The proof of Theorem \ref{thm:lb} appears in Section \ref{lb123}.
Notice that the above theorem does not provide any bound as the
eigenvalue $\lambda$ approaches one of the edges $\pm 2\sqrt{d}$, which is 
consistent with Remark \ref{rem:tempered}.

\begin{remark}[Connections to other Notions of Delocalization] 
Various other notions of delocalization for eigenvectors have been studied --- $\ell_\infty$ delocalization
	as mentioned above, lower bounds on the $\ell_1$ norm, and ``no-gaps'' delocalization \cite{rv1,rv2, eldan2017braess} (see the surveys  \cite{vusurvey,rudsurvey} for details). 
Note that taking $\e=\frac{C\log_d\log_d(n)}{\log_{d}(n)}$ for a suitably large constant $C$ in Theorem \ref{thm:main} and $k=1$ implies an $\ell_\infty$ bound of $O(\sqrt{\frac{\log_d\log_d(n)}{\log_d(n)}})$ for any eigenvector of any $d+1-$regular graph with girth $\Omega(\log_d(n)).$ 
Moreover	the examples we construct in Theorem \ref{thm:lb} show that one cannot expect to do much better. This is a much weaker result than the known bounds for random $d+1-$regular graphs where the corresponding bound is $\tilde O\left({\frac{1}{\sqrt n}}\right)$ suppressing logarithmic terms (see \cite{yau}).  
This establishes that the delocalization properties of high girth graphs are weaker than those of random graphs.
\end{remark}

\begin{remark}\label{ana1}  {In \cite{anantharaman}, the authors prove a quantum ergodicity result (``most"
 eigenfunctions are delocalized in some quantitative sense) for large regular graphs which are expanders (i.e. the spectral gap of the associated random walk operator is bounded away from zero uniformly in the size of the graph) and have a few short cycles. A canonical model satisfying the above two properties are  uniformly chosen random regular graphs. A discussion about the results of this paper in the context of the results in \cite{anantharaman} is presented in Remark \ref{ana2}.}
\end{remark}

\subsection{Connection between localization and low girth : $\e=1$ case.}
Before proceeding to the proofs of these theorems, we give a quick proof of the upper bound in the extreme case $\e=1$, i.e., when the entire mass is supported on a small set, to give some intuition
about why a localized eigenvector implies a short cycle. Assume $G$ is a $(d+1)$-regular graph with adjacency matrix $A$ and $Av=\lambda v$ for 
a vector $v$ with exactly $k$ nonzero entries. Let $H$ be the induced subgraph of $G$ supported on the nonzero vertices. Observe
that for the eigenvector equation to hold for any vertex $s\notin H$, we must have 
 {$$\sum_{t\in V}v(t)A(s,t)=\sum_{t\in H}v(t)A(s,t)=0,$$}
so in particular any such $s$ must have at least two neighbors (of opposite signs for the value of $v$) inside $H$. Thus, for every edge $ts$ with $t\in H$ leaving $H$, there must be some $t'\in H$ such that $tst'$ is a path of length $2$ in $G$. Replace all such paths by new edges $tt'$ to obtain a graph $H'$ on the vertices of $H$ (possibly creating multi-edges). Note that  {every vertex in this graph has degree at least $(d+1)$} ( {the degree can in fact exceed $d+1$ if one edge  $ts$ is a part of many paths $tst'$}). Now, if $H'$ has girth $g$, then any ball
of radius $g/2-1$ does not contain cycles. Growing a ball from any vertex, we find that
$$ d^{g/2-1}\le |H'|\le k,$$
which implies that $g\le 2\log_d(k)+2.$ Since every edge in $H'$ corresponds to a path of length at most $2$ in $G$, $G$ must contain
a cycle of length at most $4\log_d(k)+4$.

Theorem \ref{thm:main} shows that this continues to happen even when
$\e=o(1)$. Note that since the girth of a $(d+1)$-regular graph on $n$
vertices is at most $O(\log_d(n))$ by a similar argument, the only interesting
regime is when $\e=\Omega(1/\log_d(n))$.

\section{Improved Upper Bound}\label{sec12ub}
In this section we prove Theorem \ref{thm:main}, at a high level following the approach of \cite{bl}.
The main ingredient is the following hyper-contractivity estimate. Just for completeness we include the following  definition of hyper-contractivity:
  {For any $1\le p < q \le \infty,$ an operator $S$ from $\ell_p$ to $\ell_q$ is said to be hyper-contractive if its operator norm $\|S\|_{p\to q}$ is bounded by $1$. }

Let $T_m$ be the Chebyshev polynomials of the first kind, i.e., $T_m(\cos \theta)=\cos (m\theta).$
\begin{lemma}[Hypercontractivity of Chebyshev Polynomials, \cite{bl}]\label{lem:cheby} If $A$ is a $d+1$-regular graph with girth $g$, then for all even $m<g/2$,
	$$\left\|T_m\left(A/(2\sqrt{d})\right)\right\|_{1\rightarrow\infty}= \frac{d-1}{2d^{m/2}}.$$	
\end{lemma}
The proof appearing in \cite{bl} is based on a spectral decomposition in terms of spherical functions on trees. For completeness we give a quick proof of the above using connections to non-backtracking walks instead. 
\begin{proof}
Let $U_m(\cdot)$ defined by $U_m(\cos \theta)=\frac{\sin((m+1)\theta)}{\sin \theta}$ be the Chebyshev polynomials of the second kind. 
It is well known that  for any $m,$ (see for e.g. Section 2 in \cite{noga}) 
$$B^{(m)}=d^{m/2}\left(U_m(A/(2\sqrt d))-\frac{1}{d}{U_{m-2}(A/(2\sqrt d))}\right),$$ where for any pair of vertices $u,v \in V$, the entry
$B^{(m)}(u,v),$ is the number of non-backtracking walks of length $m$ between $u$ and $v$. 
At this point we use the following well known relation between the Chebyshev Polynomials of the first and second kind:
 $T_m=\frac{1}{2}\left(U_m-U_{m-2}\right).$ Putting the above together we get 
 \begin{equation}\label{recur1}
 T_{m}(A/(2\sqrt d))=\frac{1}{d}\left(T_{m-2}(A/(2\sqrt d))\right)+\frac{1}{2}\bigl(\frac{B^{(m)}}{d^{m/2}}-\frac{B^{(m-2)}}{d^{m/2-1}}\bigr).\end{equation}
Now note that $\left\|T_m\left(A/(2\sqrt{d})\right)\right\|_{1\rightarrow\infty}$ is nothing but the maximum entry of the corresponding matrix. 
Since $m<g/2$ by hypothesis, for all $j\le m$ and for all $u,v \in V$ we have $B^{(j)}(u,v)=\delta_{j,\dist(u,v)}$ where $\dist(u,v)$ is the graph distance between $u$ and $v.$  Summing \eqref{recur1} over $2, 4,\ldots, m$ after multiplying both sides of the equation corresponding to $m-2j$ by $\frac{1}{d^j}$ and using the last observation completes the proof.
\end{proof}

Using the above lemma, the next approximation result is at the heart of the proof of Theorem \ref{thm:main}. As will be clear soon, given any eigenvalue $\lambda_0$ of $A/(2\sqrt d)$, the proof of Theorem \ref{thm:main} demands the existence of a polynomial $f$, with the following two properties:
\begin{enumerate}
\item $f(A/(2 \sqrt d))$ is hyper-contractive. 
\item $f(\lambda_0)$ is large, and $f(\lambda)$ is not too negative for any other eigenvalue $\lambda.$ 
\end{enumerate} 
The key insight then is that $f(A/(2\sqrt d))$ in some approximate sense acts as a projector onto the $\lambda_0$-eigenspace of $A/2\sqrt{d}$, and at the same time is hyper-contractive. By analyzing the action of the operator $f(A/(2\sqrt d))$ on the corresponding
eigenvector one can then show that  the latter cannot be localized. 
The following lemma states that such a polynomial exists. It is in the proof of this lemma that we achieve the required estimates needed to improve the bounds in \cite{bl}.
\begin{lemma}[Hypercontractive Polynomial Approximation]\label{lem:kphi} If $A$ has girth $g$, then for all positive integers $r,m$ such that $r$ is even, $mr<g/2$, and   {all} $\lambda\in \R$ there exists a polynomial $f$ such that:
	\begin{enumerate}
		\item $f(\lambda)\ge \frac{m}{4}-1$.
		\item $f(x)\ge -1$ on $\R$.
		\item $\|f(A/2\sqrt{d})\|_{1\rightarrow\infty}\le \frac{2(d-1)}{d^{r/2}}$.
	\end{enumerate}
\end{lemma}
\begin{figure}
	\centering
	\includegraphics[scale=.4]{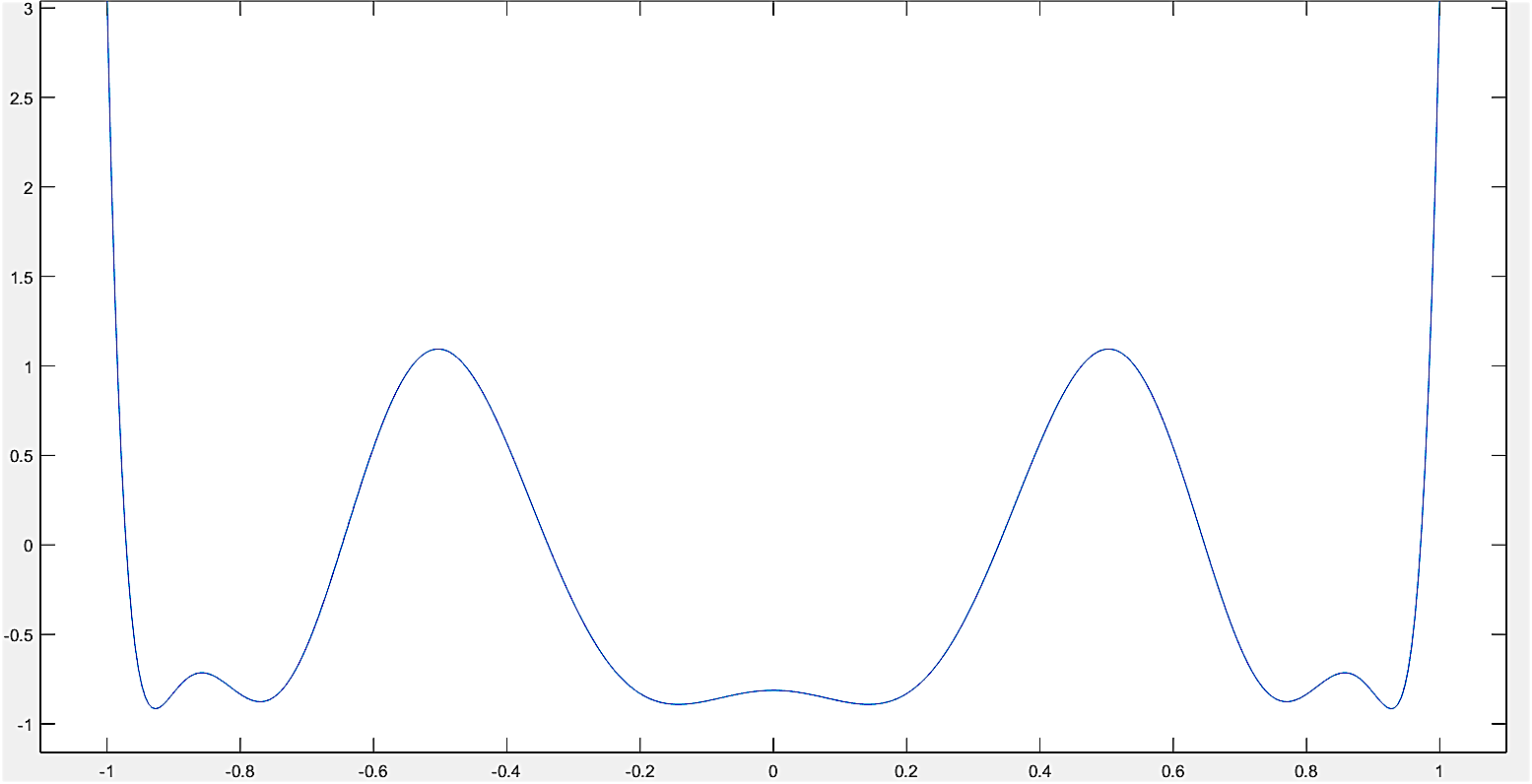}
	\caption{An example of the polynomial $f$ in Lemma \ref{lem:kphi}, with parameters $m=8,r=2,\phi=\pi/3$}
\end{figure}

\begin{proof}Assume first that $\lambda\in [-1,1]$. We will use the Fejer kernel of order $m,$
	$$F_m(\theta):=\sum_{j=-m}^m (1-|j|/m)e^{ij\theta}=1+2\sum_{j=1}^m(1-j/m)\cos(j\theta).$$
	 Recall that $F_m(\theta)\ge 0$ and $F_m(0)=m$.
	 Let $\lambda=\cos \phi$ for $\phi\in [0,\pi]$ and define
	$$ K_\phi(\theta):=\frac12\left(F_m(r(\theta-\phi))+F_m(r(\theta+\phi))\right)-1,$$
	and notice that $K_\phi(\theta)\ge -1$ and,
	\begin{equation}\label{eqn:kbig}K_\phi(\phi)\ge \frac{1}{2} (F_m(0)+0)-1=m/2-1,\end{equation}
	and
	\begin{align*}
		K_\phi(\theta) &= \sum_{j=1}^m(1-j/m)\cos(jr(\theta-\phi))+\cos(jr(\theta+\phi))\\
		&=\sum_{j=1}^m 2(1-j/m)\cos(jr\phi)\cos(jr\theta)\\
		&=2\sum_{j=1}^m (1-j/m)\cos(jr\phi)T_{jr}(\cos(\theta)).
	\end{align*}
	Let
	$$f(x):=\sum_{j=1}^m (1-j/m)(\cos(jr\phi)+1)T_{jr}(x),$$
	so that 
	$$f(\cos(\theta))=\frac{K_\phi(\theta)+K_0(\theta)}{2}.$$
	The first property is implied by \eqref{eqn:kbig} and $K_0(\theta)\ge -1$:
	$$f(\lambda)=K_\phi(\phi)/2+K_0(\phi)/2\ge \frac{m}4-1.$$
	The second property holds for $x=\cos(\theta)\in [-1,1]$ since $K_\phi(\theta),K_0(\theta)\ge -1$.
	For $x\notin [-1,1]$ we observe that $f$ is a nonnegative linear combination of Chebyshev
	polynomials of even degree, which are nonnegative outside $[-1,1]$.
	For the third property, we observe that:
	
	\begin{align*}\|f(A/2\sqrt{d})\|_{1\rightarrow\infty} 
		&\le \sum_{j=1}^m |(1-j/m)(\cos(jr\phi)+1)|\|T_{jr}(A/2\sqrt{d})\|_{1\rightarrow\infty}
		\\&\le 2(d-1)\left(\frac{1}{2d^{r/2}}+\frac1{2d^{2r/2}}+\ldots\right)
		\\&\le\frac{2(d-1)}{d^{r/2}},\end{align*}
	by Lemma \ref{lem:cheby}, 	as desired.

	If $\lambda\notin [-1,1]$, then we simply use the polynomial $f$ corresponding to $\lambda=1$ (which by symmetry
	is the same as the one for $\lambda=-1$). Properties (2) and (3) continue to hold, and property (1) holds
	because $f$ is a nonnegative linear combination of even degree Chebyshev polynomials, which are increasing on $[1,\infty)$
	and decreasing on $(-\infty,1]$. Thus for such an $f$ and $\lambda$ it follows that $f(\lambda)\ge f(1)=f(-1)\ge \frac{m}{4}-1.$
\end{proof}
We now finish the proof of Theorem \ref{thm:main}.
	\begin{proof}[Proof of Theorem \ref{thm:main}]
		Let $\lambda$ be an eigenvalue of $A/2\sqrt{d}$ with normalized eigenvector $v$. Let $f$ be the polynomial from Lemma \ref{lem:kphi} applied
		to $\lambda$, $m=\lceil 4/\e\rceil +4$, and $r=\lceil g/2m\rceil -1$ or $\lceil g/2m\rceil-2$, whichever is even. Taking $K=f(A/2\sqrt{d})$,
	we then have:
	\begin{equation}\label{eqn:upper}
	\langle v1_S,K v1_S\rangle \le \|K\|_\oi \|v1_S\|_1^2\le 2d\cdot d^{-\e g/4+1}\cdot\|v1_S\|_1^2\le 2d\cdot d^{-\e g/4+1}\cdot |S|\|v1_S\|_2^2 = 2d^{2-\e g/4} |S|\e,
	\end{equation}
	 since $r\ge \e g/8-2$, by Property (3) of Lemma \ref{lem:kphi} and $\|v1_S\|_1^2\le |S|\|v1_S\|_2^2$ (Cauchy-Schwarz inequality).

	On the other hand, decompose $v1_S$ as $av+bw$ where $w$ is a unit vector orthogonal to $v$ and $a,b$ are scalars. Observe that
	$$a=\langle v1_S,v\rangle=\|v_S\|^2=\e,$$
	and
	$$b^2 = \|v1_S\|^2-a^2=\e(1-\e).$$
	Since $\langle v, Kw\rangle=0$, we have: 
	\begin{align*}
		\langle v1_S, K v1_S\rangle &= a^2\langle v, K v\rangle + b^2\langle w, K w\rangle 
		\\& \ge a^2(1/\e) - b^2 \quad\textrm{by (1) and (2) of Lemma \ref{lem:kphi}}
		\\&= \e-\e(1-\e)=\e^2.
	\end{align*}
Combining this with \eqref{eqn:upper}, we obtain:
	$$|S|\ge \frac{d^{\e g/4}\e}{2d^2},$$
	as desired.
\end{proof}
 \begin{remark}[Improvement in the Untempered Case]\label{improve43}
	 The proof of Lemma \ref{lem:kphi} is clearly wasteful with regards to
	 eigenvalues of $A$ outside $[-2\sqrt{d},2\sqrt{d}]$; in particular,
	 noting that Chebyshev polynomials blow up exponentially outside
	 $[-1,1]$, one can considerably improve the approximation bound for
	 untempered $\lambda$ which are bounded away from the edge values $\{-2\sqrt{d},2\sqrt{d}\},$ (see Remark \ref{rem:tempered}), and obtain a
	 significantly stronger delocalization result in this case.
 \end{remark}

  { \begin{remark}\label{nogirth} Note that in the previous arguments, the
high girth assumption was only used in Lemma \ref{lem:cheby}, to prove a bound
on $\left\|T_m\left(A/(2\sqrt{d})\right)\right\|_{1\rightarrow\infty}$. However
the proof of the lemma shows that there is a lot of room to relax our
	assumptions. For example,   {let $G$} be a random regular graph of size $n$ and
	degree $d+1,$ a model also considered in \cite{anantharaman},   {for which} it is known
(for e.g. see \cite[Lemma 2.1]{lsregular}) that with high probability, there
exists at most one cycle in the $\rho=\frac{1}{5}\log_{d-1} n$ neighborhood of
any vertex. This implies that for any $j\le \frac{\rho}{2}$ and vertices
$u,v$, the number of non-backtracking walks from $u$ to $v$ of length $j$
($B^{(j)}(u,v)$ according to the notation in Lemma \ref{lem:cheby}) is at most
$2$. One can check that this implies Lemma \ref{lem:cheby} holds with the bound
$ \frac{d-1}{2d^{m/2}}$  replaced by  $\frac{O(m)}{d^{m/2}}.$ This 
yields a similar lower bound on $|S|$ as in the conclusion of Theorem
\ref{thm:main}, up to certain polynomial factors in $\e.$\end{remark}}

\section{Lower Bound}\label{lb123}
In this section, we prove the Theorem \ref{thm:lb}, which shows that the logarithmic dependence on $k$ and polynomial dependence on $\e$
in Theorem \ref{thm:main} are sharp up to a $\log(1/\e)$ term.

The starting point is to observe that certain eigenvectors of finite trees
already have good localization properties. For the remainder of the section, we
will refer to a complete tree of finite depth $D$ (i.e., $D+1$ {\em
levels} of vertices including the root which corresponds to level $0$) in
which every non-leaf vertex has degree $d+1$ as a {\em $d-$ary tree}. We will
pay special attention to eigenvectors of the adjacency matrices of $d-$ary
trees which are {\em radial}, which means that they assign the same value to
vertices in a given level.

We begin by recording some facts about eigenvalues and eigenvectors of  {rooted} $d-$ary
trees. Recall that the eigenvalues of a $d-$ary tree are contained in the
interval $(-2\sqrt{d},2\sqrt{d})$ \cite[Section 5]{hoory}. For our purposes we will
only consider radial eigenvectors, and we will refer to the corresponding eigenvalues as {\em radial eigenvalues}.

\begin{lemma}[Radial Eigenvalues]\label{sym1} For any positive integer $D\ge 2,$ $A_k$ the adjacency matrix of $T_D,$ a $d-$ary tree of depth $D,$ has exactly $D+1$ radial eigenvalues counting multiplicities. 
\end{lemma}

\begin{proof} Let the vector space of all the radial functions on the vertices of $T_D$ be called $\mathscr{S}_{\rm sym}.$ Clearly $\mathscr{S}_{\rm sym}$ has dimension $D+1.$  Let $n$ be the total number of vertices in $T_D.$ Thus $\mathscr{S}_{\rm sym}\subset \C^n.$ Now notice that $A_D$ keeps $\mathscr{S}_{\rm sym}$ invariant and because it is a self-adjoint operator  the same is true for its orthogonal complement $\mathscr{S}_{\rm sym}^{\perp}$. 
The proof now follows by conjugating $A_D$ by an orthogonal matrix $U$ whose first $D+1$ rows are formed by an orthonormal basis of $\mathscr{S}_{\rm sym}$ and the remaining rows are formed by 
	an orthonormal basis of $\mathscr{S}^{\perp}_{\rm sym}.$ This transforms $A_D$ into a block diagonal self adjoint matrix with two blocks of size $D+1$ and $n-(D+1)$ respectively.  The proof is now complete by standard spectral theory of self-adjoint operators.

\end{proof}

\begin{lemma}[Eigenvalues of $d-$ary Trees]\label{lem:dense} The set of all  {radial eigenvalues}
	 of any infinite sequence of distinct finite $d-$ary trees is dense in the interval
	$(-2\sqrt{d},2\sqrt{d})$.
\end{lemma}
\begin{proof} Let $T_1,T_2,\ldots,T_m,\ldots$ be an infinite sequence of $d-$ary trees.
	Let $T$ be the infinite $(d+1)-$regular tree with root $r$
	and observe that there are sets $S_1\subset S_2,\ldots$ such that $T_m$ is the induced
	subgraph of $T$ on $S_m$. Let $A_m$ be the adjacency matrix of $T_m$ and let $A$ be the adjacency matrix of $T$.
	Assume for contradiction that there is a closed interval
	$$I=[\lambda-\eta,\lambda+\eta]\subset (-2\sqrt{d},2\sqrt{d})$$ such
	that every $A_m$ has no  {radial} eigenvalues in $I$.
	
	We derive a contradiction with the fact \cite[Theorem 5.2]{hoory} that for $\lambda\in\mathrm{spec}(A)=[-2\sqrt{d},2\sqrt{d}]$,
	there is no vector $v\in\ell_2$ such that $(\lambda I - A_T)v=e_r$, where $e_r$ is the indicator of $r$. 
	Our assumption along with the invariance of $\mathscr{S}_{\rm sym},$ implies that  the operator $A_{\rm sym}(\lambda)$ which is the restriction of the operator $(\lambda I - A_{m})$ on the subspace $\mathscr{S}_{\rm sym}$ is invertible, and  
	$\|A_{\rm sym}(\lambda)^{-1}\|\le \eta^{-1},$ 
	for all $m.$ Now since $P_{m}e_r$ is a radial vector where $P_m:\ell_2\rightarrow \C^{S_m}$ is the restriction onto $\C^{S_m}.$, there must be a sequence of finite 
	dimensional radial vectors $\{v_m\}$ with $v_m\in\C^{S_m}$ such that 
	$$A_{\rm sym}(\lambda)v_m=P_mA_{\rm sym}(\lambda)P_m^Tv_m=P_me_r.$$
	Moreover, we have the uniform bound
	$$\|P_m^Tv_m\|=\|v_m\|\le \eta^{-1}$$ for all $m$. The Banach-Alaoglu Theorem implies that $\{P_m^Tv_m\}$ must have a weakly
	convergent subsequence; let $v\in\ell_2$ be the weak limit of this subsequence, and note that $v$ must be radial as well.
Moreover, $v$ must satisfy:
\begin{align*}	
 e_j^T(\lambda I - A)v &= \lim_{m\rightarrow\infty} e_j^TP_m^TP_m(\lambda I - A)P_m^Tv_m\\
&= e_j^TP_m^TA_{\rm sym}(\lambda)v_m \\
&= \lim_{m\rightarrow\infty} e_j^TP_m^TP_me_r=e_j^Te_r,
 \end{align*}
	for every vertex $j\in T$, so in fact we must have $(\lambda I - A)v=e_r$, which is impossible.
\end{proof}

The next lemma lower bounds the relative mass of an eigenvector on different levels of a tree. The key reason behind such a result is the propagation of mass across levels via the eigenvalue equation. 

\begin{lemma}[Eigenvectors of $d-$ary Trees]\label{lem:eigvec} Assume $d\ge 2$
	and let $T$ be a $d-$ary tree of depth $D$ with root $r$. Let $S_0=\{r\},S_1,\ldots,S_D\subset
	T$ be the vertices at levels $0,1,\ldots, D$ of the tree and let $v$ be a
	radial eigenvector of its adjacency matrix with eigenvalue 
	$\lambda=2\sqrt{d}\cos\theta\in (-2\sqrt{d},2\sqrt{d})$. Then every
	pair of adjacent levels has approximately the same total $\ell_2^2$ mass as the root:
	$$\Omega(\sin^2\theta) = \frac{\|v_{S_i}\|_2^2+\|v_{S_{i+1}}\|_2^2}{\|v(r)\|_2^2}=O(1/\sin^{2}\theta).$$
\end{lemma}
\begin{proof} Suppose $v$ has value $x_i$ for all vertices in $S_i$, and for convenience 
	assume that the root has value $x_0=1$ (although this makes $v$ un-normalized). 
	The eigenvector equation at the non-leaf vertices yields the following quadratic recurrence:
	\begin{align*} \lambda x_0 &= (d+1)x_1,\\
	 \lambda x_i &= x_{i-1}+dx_{i+1}\quad 1\le i\le D-1,\\
	 \end{align*}
	which must be satisfied by any eigenvector (ignoring the boundary condition at the leaves).
	Since we are interested in the total $\ell_2^2$ mass at each level, it will be more convenient
	to work with the quantities 
	$$m_0=x_0=1\qquad\textrm{and}\qquad m_i = \sqrt{|S_i|}x_i = \sqrt{(d+1)d^{i-1}}x_i\quad 1\le i\le D,$$
	which satisfy $m_i^2 = \|v_{S_i}\|_2^2$. Rewriting the recurrence in terms of the $m_i$, we obtain:
	\begin{align*} m_1 &= \frac{\lambda}{\sqrt{d+1}} m_0,\\
		m_{i+1} &= \frac{\lambda m_i}{d}\cdot\sqrt{\frac{|S_{i+1}|}{|S_i|}} - \frac{m_{i-1}}{d}\cdot\sqrt{\frac{|S_{i+1}|}{|S_{i-1}|}} = \frac{\lambda}{\sqrt{d}}m_i - m_{i-1},\quad 1\le i\le D+1.
	 \end{align*}
	Letting $\lambda=2\sqrt{d}\cos\theta$ and writing the above in matrix form, we have
	$$w_{i+1}:=\bm{ m_{i+1} \\ m_i} = \bm{2\cos\theta & -1 \\ 1 & 0}\bm{m_i\\m_{i-1}} =: PDP^{-1}w_i,$$
	since the matrix above is diagonalizable for $\theta\neq 0,\pi$, with
	$$ D:= \bm{e^{i\theta} & 0\\ 0 & e^{-i\theta}},\quad P = \bm{1 & 1\\ e^{-i\theta} & e^{i\theta}}.$$
	Since $D$ is unitary we have $\|P^{-1}w_{i}\|=\|P^{-1}w_0\|$ for all $i$. Observe that
	$$\|P\|\le 2,\quad \|P^{-1}\|\le \frac{1}{\sin\theta}.$$
	Whence
	$$\frac{\|w_i\|}{\|w_0\|}\in [\frac{\sin\theta}{2},\frac{2}{\sin\theta}].$$
	Noting that $|m_1|\le 2|m_0|$ and squaring yields the claim.
\end{proof}

Let $T$ be a $d-$ary tree of depth $D$. Choosing $S$ to be the top $\lfloor \e D\cdot (d/d+1)\rfloor$ levels of $T$ and applying the above lemma to any eigenvector with eigenvalue bounded away from $\pm 2\sqrt{d}$, we find that $\|v_S\|_2^2=\Theta(\e)$ and
$|S|=O((d+1)^{\e D})=O(n^\e)$, where $n$ is the number of vertices
in the tree. This is exactly the kind of localization we want for our lower bound.
Unfortunately, finite $d-$ary trees are not regular because they
have leaves. The rest of this section is devoted to showing that we can
nonetheless embed these trees in $(d+1)-$regular graphs without disturbing
their eigenvectors or creating any short cycles, thereby establishing Theorem
\ref{thm:lb}. The main device in doing this is the following lemma which shows that it is possible
to identify the leaves of two trees in a manner which does not introduce short cycles.

\begin{lemma}[Pairing of Trees]\label{lem:pairing} Suppose $T_1$ and $T_2$ are two $d-$ary trees of depth $D$, each with
	$n=(d+1)d^{D-1}$ leaf vertices $V_1$ and $V_2$. Then for sufficiently large $n$ there is a bijection $\pi:V_1\rightarrow V_2$
	such that the graph obtained by identifying $v$ with $\pi(v)$ (i.e., the new graph is obtained by first adding an edge between a vertex $v$ and $\pi(v)$ for $v\in V_1$ and then contracting the added edges.) has girth at least $\log_d(n)/4$\end{lemma}
\begin{figure}
\centering
\includegraphics[scale=.6]{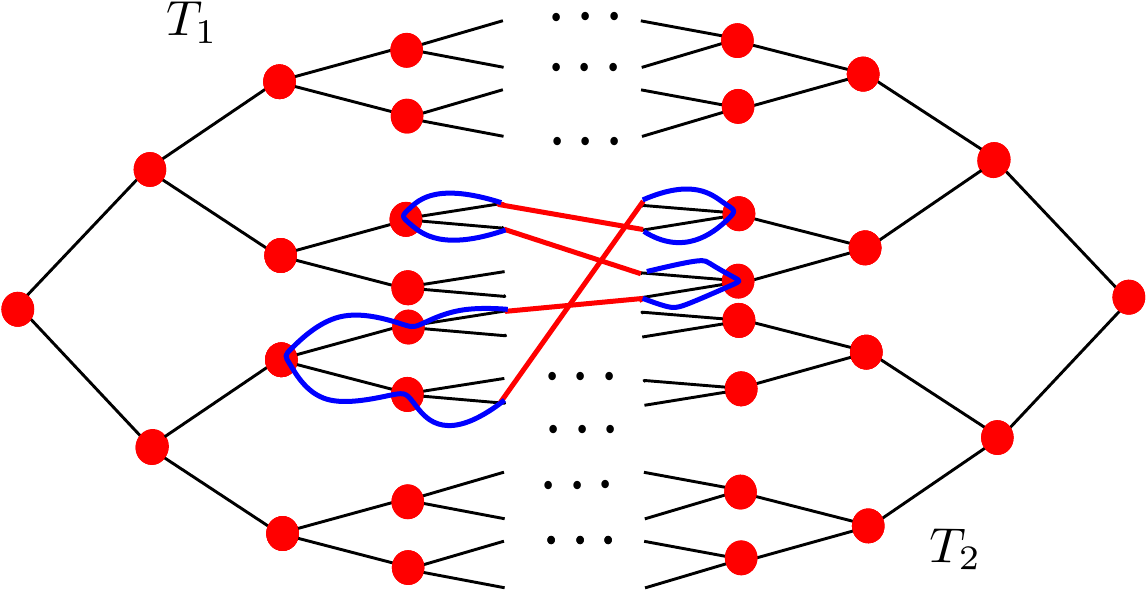}
\caption{This describes the construction of gluing two 3- regular trees by a random matching of the leaves. A cycle in the glued graph is illustrated where the blue paths denote the excursions into the trees and the red edges denote the jump from one tree to the other. However in the picture note that all the interior vertices in the trees  have degree three but the roots only have degree two (since the third edge is not significant for the purposes of the illustration, it was omitted just to avoid cluttering in the figure).}
\label{pairing}
\end{figure}
We defer the proof of this lemma. 
However such a graph obtained is still not regular. The second ingredient is:
\begin{lemma}[Degree-Fixing Gadget]\label{lem:gadget} For every degree $d\ge 3$ and sufficiently large $n$, there is a graph $H$ on $n$ vertices with the following properties:
	\begin{enumerate}
		\item $H$ has one distinguished vertex of degree $d-2$ and the remaining vertices have degree $d$.
		\item $H$ has girth at least $\log_{d-1}(n)/3$.
	\end{enumerate}
\end{lemma}

\begin{proof} According to Corollary 2 of \cite{MWW}, the number of $d-$regular graphs with girth at least $g$ is asymptotic to
	$$ \frac{(dn)!}{(dn/2)!2^{dn/2}(d!)^n}\cdot\exp\left(-\sum_{r=1}^g\frac{(d-1)^r}{2r}+o(1)\right) \ge \exp\left(dn/2-g(d-1)^r+o(1)\right),$$
	whenever $(d-1)^{2g-1}=o(n).$ Taking $g=\log_{d-1}(n)/3+1$ we find that
	this condition is satisfied and the right hand side is positive for
	large enough $n$.  Let $G$ be a $d-$regular graph on $n$ vertices with
	girth at least $g$. Let $v$ be any vertex of $G$ and let $u_1,u_2$ be
	two of its neighbors. Let $H$ be the graph obtained by deleting the
	edges $vu_1$ and $vu_2$ and adding the edge $u_1u_2$. Observe that $v$
	has degree $d-2$ and every other vertex has degree $d$ in $H$.
	Moreover, since we replaced  a path of length $2$ by an edge, the length
	of every cycle decreases by at most $1$, so $H$ must have girth at
	least $\log_{d-1}(n)/3$, as desired.
\end{proof}

Equipped with the above lemmas we now complete our construction and hence prove Theorem \ref{thm:lb}.

\begin{proof}[Proof of Theorem \ref{thm:lb}] Let $d\ge 2$, and let $k$ be any integer larger than the $n$ required for Lemmas \ref{lem:pairing} and \ref{lem:gadget} to apply. Suppose $\e\in (0,1)$ is given. Let $t$ be the largest integer such that $(d+1)d^{t-1}\le k$.
	Choose $D-1=\lceil t/\e\rceil$ and let $T_1$ and $T_2$ be two disjoint $d-$ary trees with
	$D-1$ levels. Let $S_1$ and $S_2$ be the sets of vertices consisting of the top $t$ levels of $T_1$ and $T_2$
	respectively.
	
	Let $\lambda$ be an eigenvalue of $A_{T_1}$ and let $f$ be the
	corresponding normalized radial eigenvector, ($f$  has the same value on every vertex within a level) as in Lemma \ref{lem:eigvec}. By
	Lemma \ref{lem:eigvec}, we know 
	$$\|f_{S_1}\|_2^2=\|f_{S_2}\|_2^2=\Omega_\lambda(\e),$$
	where the implicit constant is $\Omega(\sin^4\theta)$ for $\lambda=2\cos\theta$.

	Construct $T_1'$ and $T_2'$ by attaching $d$ new {\em marked} vertices to each
	leaf of $T_1$ and $T_2$, respectively, so that they are $d-$ary trees of depth
	$D$, with $n=(d+1)d^{D-1}$ leaves each, corresponding to the marked vertices. 
	Apply Lemma \ref{lem:pairing} to pair these marked leaves; call the resulting graph $H$.
	Notice that $H$ is $d+1$-regular except for the marked vertices, which have degree two.  Applying
	Lemma \ref{lem:gadget} with degree parameter $d+1$ and size $n$, we obtain a disjoint collection of graphs $W_v$ (one for every marked vertex $v$) each with 
	a single distinguished vertex of degree $d-1$, all remaining vertices of degree $d+1$, and girth $\log_d(n)/3.$ Finally,
	let $G$ be the graph obtained by identifying each marked vertex $v$ with the distinguished vertex of $W_v.$
.

	Observe that $G$ has girth at least $\log_d(n)/4=\Omega(D)=\Omega(\log_d(k)/\e)$, since $H$ has girth at least this much by Lemma \ref{lem:pairing} and
	attaching disjoint copies of $W$ at single vertices does not create any new cycles.

Using the symmetry in the above construction we now prescribe an eigenvector of $G.$ Let $\nu$ be the function equal to $f$ on vertices of $T_1$, $-f$ on vertices of $T_2$, and zero elsewhere. We claim
	that $\nu$ is an eigenvector of $G$ with eigenvalue $\lambda$. 
	To see this one has to verify the eigenvector equation at vertices of three kinds:
	\begin{itemize}
	\item At every vertex of $T_1$ and $T_2$ because all new neighbors of those vertices are assigned a value of $0$ in $\nu$.
	\item It is also satisfied at the marked vertices, because every such
		vertex is adjacent to exactly one leaf in $T_1$ and one leaf in $T_2$.
	which have the same values with opposite signs,
	\item The remaining vertices in copies of $W$ have value zero, so the eigenvector equation is trivially satisfied.
	\end{itemize}	
	 
	Observing that $\|\nu_{S_1\cup S_2}\|_2^2=\Omega_\lambda (\e)$ with $|S_1\cup S_2|=2k$ finishes the construction.
	Since this construction is valid for infinitely many $n$,  { Lemma \ref{lem:dense} implies that the set of radial eigenvalues
	 is dense in $(-2\sqrt{d},2\sqrt{d})$, as desired.}

\end{proof}
\begin{figure}
\centering
\includegraphics[scale=.7]{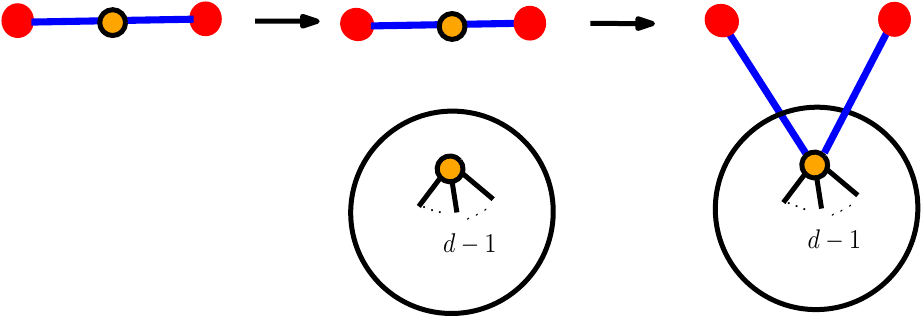}
\caption{This illustrates the use of the degree correcting gadget.  The two red vertices in the first figure denote two leaf vertices in $T_1$ and $T_2$ respectively  connected via the identified marked vertex (see Figure \ref{pairing}). The marked  vertex (yellow)  is then identified with a degree $d-1$ vertex in the gadget.}
\label{gadget}
\end{figure}

\begin{remark}\label{ana2}  {Comparing the results in  this paper to the ones in \cite{anantharaman}, we wish to make the following two remarks. 

$\bullet$The results in \cite{anantharaman} imply that in a certain
	quantitative way, all but a vanishing fraction of eigenfunctions are
	close to a constant vector in a weak sense. In contrast, in this
	article we show that the radial eigenvectors of
	finite trees lead to localized eigenvectors for
	certain $d+1$ regular graphs with high girth. However, as discussed in
	Lemma \ref{sym1}, the number of such eigenvectors is rather small,
	 {and constitutes a vanishing fraction}.  There could be other
	non-radial localized eigenvectors as well, but as the results in
	\cite{anantharaman} show, if the graph $G$ is an expander and has few
	short cycles, such eigenvectors must form a vanishing fraction of all
	eigenvectors. }

$\bullet$ Note that in the above proof, even though the graph $H$ obtained from
	gluing $T_1$ and $T_2$  has nice expansion properties, the final graph
	$G$ obtained from $H$ by adding the degree correcting gadgets are not
	expanders since the gadgets only have an edge boundary of size $2$.
	 {We outline a more complicated alternate approach which could
	conceivably produce high girth expanders in Section \ref{expan1}, after
	the proof of Lemma \ref{lem:pairing}, but do not pursue this in the present paper.}
 \end{remark}

\subsection{Proof of Lemma \ref{lem:pairing}}
\noindent {\em Probabilistic Model.} Our construction is probabilistic, and inspired by the
switching argument of \cite{MWW}, but with much cruder estimates since we are
not interested in precise asymptotic enumeration, but only in showing that a
certain probability is not zero. 
Let $T_1$ and $T_2$ be two $d-$ary trees of depth $D$, each with exactly
$n=(d+1)d^{D-1}$ leaf vertices, henceforth denoted $V_1$ and $V_2$. Consider
the random graph $G$ obtained by taking the union of $T_1$ and $T_2$ and a  {random} perfect
matching between $V_1$ and $V_2$ ( {the edges induced by the matching will henceforth be called matching edges}). \\

\noindent {\em Graph-Theoretic Terminology.} A {\em cycle} in a graph is an oriented closed walk with no repeated edges. We
will consider cyclic shifts and reversals of a cycle to be the same cycle. 
 {Note that because no edge is repeated, 
a cycle in $G$ has a natural decomposition into a alternating sequence of paths and matching edges (will be called matching edge traversals to incorporate the orientation as well), where the paths are entirely a subset of $T_1$ or $T_2$ alternately.} Naturally we will call such paths as tree excursions. We will formally denote this as a sequence of alternating {\em
matching edge traversals} $e_i$ and {\em tree excursions} $\gamma_i$ (see Figure \ref{pairing}): 
$$ e_1,\gamma_1,e_2,\gamma_2,\ldots,e_{k},\gamma_k$$ 
(or equivalently
$\gamma_1,e_2,\gamma_2,\ldots,e_{k},\gamma_k, e_1$),
where the $\gamma_i$ are simple paths in either $T_1$ or $T_2$ with endpoints at leaves, and $\gamma_k$ ends where
$e_1$ begins. We will follow the convention that $\gamma_1,\gamma_3,\ldots$ are
excursions in $T_1$ and $\gamma_2,\gamma_4,\ldots$ are excursions in $T_2$. The total number 
of edges in a cycle will be called its {\em length}.

We begin by establishing some preliminary facts about short cycles in $G$.
\begin{lemma}[Number and Overlaps of Short Cycles]\label{lem:overlap}  Let $c<1/2$ be a constant. Then for sufficiently large $n$,
	with  {positive} probability ( {independent of $n$}) we have both of the following properties simultaneously: 
	\begin{enumerate}
	\item $G$ does not contain two cycles of  {length at most $L,$} which share a {matching} edge,
	\item $G$ contains at most $B=O(n^{c(1+o(1))})$ cycles of length at most $L$,
	\end{enumerate}
	where $L:=2c\log_d(n)$.
\end{lemma}
\begin{proof}
	Let $v\in V_1$ be a leaf vertex. We will first show that
	\begin{equation}\label{eqn:p1}\P[\textrm{$v$ occurs in $\ge 1$ cycle of length $\le L$}]=O(n^{-1+c(1+o(1))}).\end{equation}
	Call a cycle that occurs in $G$ with nonzero probability a {\em potential cycle}.
	Every potential cycle consists of $k$ matching traversals and $k$ tree
	excursions for some even $k$. Observe that every excursion of length $h$ has even length and consists
	of $h/2$ upward steps towards the root of the tree and $h/2$ downward
	steps back down to the leaves. 
	
	Given a starting vertex for the excursion, the upward steps are uniquely determined, and there are at most $d$ choices for each of the downward steps (since backtracking is not allowed, and the root has degree $d+1$). Since there are at most $n$ 
	choices for each matching traversal given one of its endpoints, the total 
	number of potential cycles containing $v$ with exactly $k$ matching traversals and excursions
	of lengths $h_1,\ldots,h_k$ is at most:
	\begin{equation}\label{eqn:exc}d^{h_1/2}\cdot n\cdot d^{h_2/2}\cdot n\ldots d^{h_k/2} = d^{(h_1+\ldots+h_k)/2}\cdot n^{k-1}\le n^{k-1+c},\end{equation}
	since the last matching traversal is determined by the starting vertex $v$.
	Every such potential cycle fixes $k$ matching edges, so the probability that it occurs in a random matching is at most
	$(n-k)!/n!$. Taking a union bound over all $k\le L$, ordered partitions $h_1+\ldots+h_k\le L-k$, and potential cycles
	with those parameters, we have
	$$\P[\textrm{$v$ occurs in a cycle of length $\le L$}]\le \sum_{k=1}^L e^{O(\sqrt{L})}k!\cdot\frac{n^{k-1+c}(n-k)!}{n!}=O(n^{-1+c(1+o(1))}).$$

	Next, we use a similar argument to estimate the probability that any {leaf} vertex $v$ occurs in more than one cycle. Fix $v\in V_1$ 
	and observe that a pair of potential cycles of length at most $L$ both containing $v$ can be specified by the following choices:
	\begin{itemize}
		\item Lengths $s,s'\le L$, matching traversal counts $k,k'\le
			L$, and tuples of excursion lengths $h_1,\ldots,h_k$
			and $h_1',\ldots,h_{k'}'$ for both cycles.
		\item The common matching edge $e$ incident to $v$ contained in both cycles.
		\item The excursions made in both cycles.
		\item The remaining $k-2$ and $k'-2$ matching edges in both cycles (noting that the final edge is not required
			once all excursions are specified).
	\end{itemize}
	Since any particular pair fixes $k+k'-1$ matching edges, the
	probability that it occurs in $G$ is at most $(n-(k+k'-1))!/n!$.
	Bounding excursions as in \eqref{eqn:exc} and taking a union bound, we
	have
	\begin{align*}\label{eqn:p2}&\P[\textrm{$v$ occurs in $\ge 2$ cycles of length $\le L$}]
	\\&\le L^4e^{O(\sqrt{L})}\cdot n\cdot n^{2c}\cdot \sup _{k,k'}\left (k!(k')! n^{k+k'-4}\cdot\frac{(n-(k+k'-1))!}{n!}\right)
		\\&=O(n^{2c(1+o(1))-2}).
	\end{align*}
Taking a union bound over all leaf vertices $v$, and observing that cycles in $G$ share a matching edge  if and only if they pass through a common leaf vertex we conclude that
	$$\P[\textrm{$G$ contains two  cycles {sharing a matching edge}, of length $\le L$}]=O(n^{2c(1+o(1))-1})=o(1).$$

For the second claim we sum \eqref{eqn:p1} over all $v\in {V_1}$ and apply Markov's inequality to obtain:
	$$\P[|V_{C}|>O(n^{c(1+o(1))})]<1/2,$$
	where $|V_{C}|$ denotes the set of vertices contained in at least one cycle of length $\le L$. Taking a union bound
	with our previous conclusion, we have that with probability $1/3$ all cycles in $G$ are {matching edge} disjoint and $|V_C|=O(n^{c(1+o(1))})$. 
	Since every cycle contains at least one vertex, this gives the second claim.
\end{proof}

\begin{proof}[Proof of Lemma \ref{lem:pairing}] Let $c=1/4$ and $L:=2c\log_d n$, and choose $n$ sufficiently large so that 
	Lemma \ref{lem:overlap} applies with $B=O(n^{c(1+o(1))})\le n^{1/3}$. Let $\Gamma$ be the set of graphs in the support of the random variable $G,$  {(i.e., the set of all the graphs which $G$ is equal to with nonzero probability)} 
	such that both conditions of Lemma \ref{lem:overlap} are satisfied. Clearly the support size is equal to the number of matchings which is $n!$.  {Now, since by Lemma \ref{lem:overlap} $\P(G\in \Gamma)>0$ (independent of $n$), it follows that,}
	\begin{equation}\label{eqn:factorial} |\Gamma|=\Omega(n!).\end{equation}
	For integers $z_2,z_4,\ldots,z_L\le B$ let
	$$\Gamma(z_2,\ldots,z_L)$$
	denote the subset of $\Gamma$ containing graphs with exactly $z_j$ cycles of length $j$, noting that
	in our model there are never any odd cycles.

	Our goal is to show that $\Gamma(0,\ldots,0)$ is not empty.
	Following \cite{MWW}, our strategy will be to establish the following two claims: Let $[B]=\{0,1,\ldots,B\}.$ 

	\begin{claim}\label{cl1} There exists a $z^*\in [B]^L$ such that
		$$|\Gamma(z^*_2,\ldots,z^*_L)|\ge \exp(\Omega(n\log n)).$$
	\end{claim}
	\begin{claim}\label{cl2}
		For every $z\in [B]^L$ such that $z_k>0$:
		$$\frac{|\Gamma(z_2,\ldots,z_{k}-1,\ldots,z_L)|}{|\Gamma(z_2,\ldots,z_k,\ldots,z_L)|}= \Omega(n^{-3c}).$$
	\end{claim}
	Iterating the above claims yields
	$$|\Gamma(0,\ldots,0)|\ge \exp\left(Cn\log n-O(\log(n))\sum_{i=2,\ldots L} z_i\right)=\exp(Cn\log n-\Omega(n^{1/3}\log^2 n))>1,$$
	so that with nonzero probability $G$ has no cycles of length at most $L$. Contracting all matching edges shrinks the length
	of every cycle by at most a factor of $2$, yielding the desired pairing.

	To establish Claim \ref{cl1}, we observe that the tuple $z^*\in [B]^L$ which maximizes $|\Gamma(z^*)|$ must have cardinality at least
	$$\frac{\Omega(n!)}{B^L}\ge\Omega\left(\exp\left(n\log n-O(n)-O(\log^2 n)\right)\right).$$

\begin{figure}
\centering
\includegraphics[scale=.5]{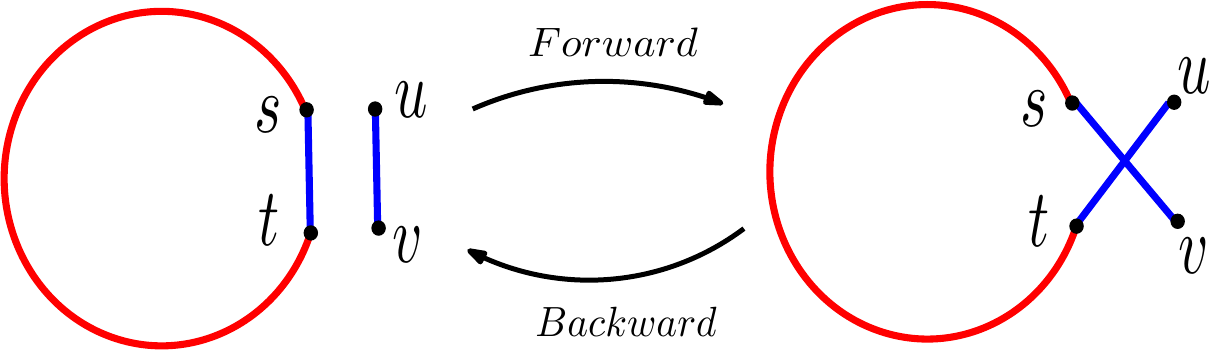}
\caption{This illustrates the switching argument: forward switching takes a cycle and an edge and produces a path while the backward switching does the reverse.}
\label{fwbw}
\end{figure}

	For Claim \ref{cl2} we use a switching argument i.e., we will switch some edges in a graph in $\Gamma(z_2,\ldots, z_k,\ldots, z_L)$ to construct a graph in $\Gamma(z_2,\ldots, z_k-1,\ldots, z_L)$. Given a graph $H\in
	\Gamma(z_2,\ldots, z_k,\ldots, z_L)$ with $z_k>0$, a {\em forward
	switching} is defined as the following operation:
	\begin{itemize}
		\item Choose the lexicographically\footnote{Fix two  { arbitrary orderings of edges and cycles}.} first matching edge $e=st$ in the lexicographically first cycle $C$ of length $k$ in $H$.
	\item Choose any matching edge $f=uv$ in $H$ at distance at least $2L$ from $e$ which is not contained in any cycle of length at most $L$.
	\item Remove $st$ and $uv$ from the matching and add $sv$ and $ut$.
	\end{itemize}
	Observe that since every {matching} edge is contained in {at most} one cycle of length $\le L$ and $f$ does not belong to any cycle of length $\le L$, removing $e$ and $f$ destroys
	only the cycle $C$ among all the cycles of length $\le L$. Since the endpoints of $e$ and $f$ are at distance $2L$ in $H$, adding $sv$ and $ut$ does not create any cycles 
	of length at most $L$. Thus, the outcome of a forward switching is a graph
	$H'\in \Gamma(z_2,\ldots, z_k-1,\ldots,z_L)$, which has exactly the same cycle counts except with one less cycle of length $k$.

	Let $\F(H)$ denote the set of foward switchings of a graph $H\in\Gamma(z_2,\ldots,z_L)$. Observe that the only choice in the switching
	is the choice of the second matching edge $f$. The number of matching edges contained in cycles of length at most $L$ is bounded by $BL=o(n)$ and the number of edges within distance $2L$ of $e$ is at most $(d+1)^{2L}=o(n)$. Therefore, for every $H\in \Gamma(z_2,\ldots,z_L)$, we have
	$$|\F(H)|=\Omega(n).$$	

	We now investigate how many forward switchings can map to a given graph $H'$ in
	$\Gamma(z_2,\ldots,z_k-1,\ldots,z_L)$. 
	Given such an $H'$, a {\em backward switching} is defined as the following operation:
	\begin{itemize}
	\item Choose two vertices $u$ and $v$ at distance exactly $k+1$ in
		$H'$, such that (a) the extreme edges $ut$ and $vs$ of the
			$uv-$path $p$ are matching edges. (b) The distance
			between $u$ and $v$ in $H'$ along any path other than
			$p$ is at least $L$.
	\item Delete the edges $ut$ and $vs$ from the matching and add edges $st$ and $uv$.
	\end{itemize}
	Observe that a backward switching always yields a graph $H\in\Gamma(z_2,\ldots,z_L)$, and that all graphs $H$
	with a forward switching equal to $H'$ may be achieved in this manner. The number of backward switchings of
	any graph $H'$ is upper bounded by
	$$|\B(H')|\le n\cdot (d+1)^{L+1} = O(n^{1+3c}),$$
	where we have overcounted by ignoring the conditions (a) and (b) in the definition of a backward switching.

	A double counting argument now yields:
	$$\frac{|\Gamma(z_2,\ldots,z_{k}-1,\ldots,z_L)|}{|\Gamma(z_2,\ldots,z_k,\ldots,z_L)|}\ge \frac{\min_H|\F(H)|}{\max_{H'}|\B(H')|}=\Omega(n^{-3c}),$$
	as desired.

\end{proof}

\subsection{Constructing expanders: a possible approach}\label{expan1} 
 {We will follow the notations in the proof of Theorem \ref{thm:lb} and assume for simplicity that $d+1=2k$ for some positive integer $k.$ Consider $T_1, T_1'$ and $T_2,T_2'$ as in the proof of  Theorem \ref{thm:lb}. For our purposes we will require $k$ copies of $T_1'$ as well as $T_2'.$
Let us call them $T'_{1,1}, T'_{1,2},\ldots T'_{1,k}$ and similarly $T'_{2,1}, T'_{2,2},\ldots T'_{2,k}.$

Now the number of leaves in $T_1'$ is equal to $dm$ where $m$ is the number of leaves in $T_1.$ It would be useful to think of the the leaves of $T'_1$ as half-edges hanging from the leaves of $T_1$ and similarly for $T_2$ and $T_2'.$
Also let us denote the natural copy of $T_1$ sitting inside $T'_{1,i}$  by $T_{1,i}$ and similarly we define the notation $T_{2,i}.$

Moreover, consider 
$(2k-1)m$ new vertices $W=\{w_1,w_2,\ldots, w_{(2k-1)m}\}$ with  $k$ half edges on each of them. 
Similarly  consider $(2k-1)m$ new vertices $W'=\{w'_1,w'_2,\ldots, w'_{(2k-1)m}\}$ with  $k$ half edges on each of them as well. 

Now as before we glue the half edges on the leaves of $\{T'_{1,i}:i=1,2,\ldots k\}$ with the half edges on the vertices in $W$ using a random matching. Note that the number of half edges incident on $\{T'_{1,i}:i=1,2,\ldots k\}$ is $kdm$ which is precisely the number of half edges incident on $W.$
Similarly we glue the half edges on the leaves of $\{T'_{2,i}:i=1,2,\ldots k\}$ with the half edges on the vertices in $W'$ using an independent random matching.

Finally to obtain a high girth graph $H,$ identify every vertex $w_i$ with the vertex $w_i'$ and call it say, $u_i.$ Thus the degree of $u_i$ is the sum of the degrees of $w_i$ and $w_i'$ which is $d+1.$ Thus the resulting graph $H$ is a $d+1$ regular graph on the vertex set formed by the union of the vertices in $\{T_{1,i}, T_{2,i}: i=1,2,\ldots,k\}$ as well as $U=\{u_1,u_2,\ldots, u_{(2k-1)m}\}.$

A localized eigenvector can now be constructed as following: pick a radial eigenvector $f$ corresponding to some radial eigenvalue $\lambda$ of $A_{T_1}.$ Now consider the function $\nu$ that is equal to  $f$ on each copy $T_{1,i}$ and is equal to  $-f$ on each copy $T_{2,i}$ and is $0$ on each $u_i.$ It can be easily checked that the $\nu$ is an eigenvector corresponding to eigenvalue $\lambda$ for the large graph $H.$ Thus the above construction avoids the gadgets. However, that the matchings indeed lead to a high girth graph with nonzero probability as in Lemma \ref{lem:pairing} still remains to be checked and will not be pursued here.  \\

The above discussion suggests that the general strategy of gluing trees could lead to various interesting constructions and investigating them and other related phenomena is left for future work. }  \\

\noindent {\bf Acknowledgment.} We would like to thank Assaf Naor and Mark Rudelson for helpful conversations, and MSRI and the Simons Institute
for the Theory of Computing, where this work was partially carried out. We are also grateful to two anonymous referees whose comments greatly improved the paper. 
\bibliographystyle{alpha}
\bibliography{bl}
\end{document}